\documentclass[reqno]{amsart}
\usepackage{enumerate,epsfig,amsmath,amssymb,amsthm,graphics,amsfonts,mathrsfs,polynom,etex,palatino,amsthm,verbatim,tikz,setspace,graphicx,array,times,romannum,multicol, hyperref,enumerate}
\usepackage[all]{xy}

\newtheorem{The}{Theorem}[section]
\newtheorem{Lem}[The]{Lemma}
\newtheorem{Cor}[The]{Corollary}
\newtheorem{Pro}[The]{Proposition}

\theoremstyle{definition}
\newtheorem{Def}[The]{Definition}
\newtheorem{Exa}[The]{Example}
\newtheorem{Rem}[The]{Remark}
\newtheorem*{Ass*}{Assumptions}

\numberwithin{equation}{The}

\title{Closed Neighborhood Ideals of Finite Simple Graphs}

\author{Jacob Honeycutt}
\address{Jacob Honeycutt, Department of Mathematics,
University of Tennessee, Knoxville,
227 Ayres Hall,1403 Circle Drive, Knoxville, TN 37996-1320
USA}
\email{jhoney12@vols.utk.edu}
\urladdr{http://www.math.utk.edu/people/bio/Jacob/Honeycutt/}

\author{Keri Sather-Wagstaff}
\address{Keri Sather-Wagstaff, School of Mathematical and Statistical Sciences,
Clemson University,
O-110 Martin Hall, Box 340975, Clemson, SC 29634
USA}
\email{ssather@clemson.edu}
\urladdr{https://ssather.people.clemson.edu/}



\begin{document}

\maketitle
\pagenumbering{arabic}

\begin{abstract}
We investigate Sharifan and Moradi's closed neighborhood ideal of a finite simple graph, which is a square-free monomial ideal in a polynomial ring over a field. We explicitly describe the minimal irreducible decompositions of these ideals. We also characterize the trees whose closed neighborhood ideals are Cohen-Macaulay; in particular, this property for closed neighborhood ideals of trees is characteristic independent. 
\end{abstract}

\section{Introduction}\label{sec210425z}
Combinatorial commutative algebra allows one to take certain combinatorial objects (e.g., graphs, simplicial complexes, and posets) and encode nontrivial combinatorial information about them algebraically. This allows one to use algebraic tools to study these objects, and it allows us to understand the algebraic objects combinatorially. 
An example of this Villarreal's~\cite{villarreal:cmg,villarreal:ma} edge ideal $I_G$ of a finite simple graph $G$, where one can see the 
irredundant irreducible decompositions of $I_G$ via certain subsets of the vertex set called vertex covers \textit{op.\ cit.}, and one can see  Cohen-Macaulay properties for certain classes of graphs, including trees \textit{op.\ cit.}\ and chordal graphs~\cite{francisco:wscmg,francisco:scmei,MR2231097}. Similar results hold for other constructions using edges and paths~\cite{morey:dcmppi,conca:msgililt,kubik:piwg,paulsen:eiwg}.

In this project we investigate another algebraic construction that sees other combinatorial aspects of $G$: the closed neighborhood ideal $N_G$, as introduced by Sharifan and Moradi~\cite{MR4132629}. It is a square-free monomial ideal whose generators are the closed neighborhoods of the vertices of $G$. See Section~\ref{sec210425a} for a complete definition and basic properties. 

Whereas the irredundant irreducible decomposition of the edge ideal of $G$ is described using the minimal vertex covers of $G$, our construction is decomposed in terms of other subsets of the vertex set that are important for studying networks: the dominating sets~\cite{MR1605684,MR1088576}. This decomposition result is Proposition~\ref{thm210426a}, the culminating result of Section~\ref{sec210425a}. 

Section~\ref{sec210425c} investigates the Cohen-Macaulay property for our construction. The main result there is Theorem~\ref{thm210426b} which characterizes the trees whose closed neighborhood ideals are Cohen-Macaulay. In particular, it shows that the Cohen-Macaulay, unmixed, and  complete intersection properties are equivalent for these ideals, and that they are independent of the characteristic of the ground field.

\begin{Ass*}
For the remainder of this paper, let $G = (V,E)$ be a
finite simple graph with vertex set $V = \{X_1,X_2,\ldots,X_d\}$ and edge set $E$ where an edge between distinct vertices $X_i$ and $X_j$ is denoted $X_iX_j$. 
(The ``simple'' assumption says that $G$ has no loops and no multi-edges. See, e.g., the text of Diestel~\cite{diestel:gt} for unexplained notions from graph theory.) 
Let $\mathbb{K}$ be a field, and consider
the polynomial ring $R = \mathbb{K}[X_1,X_2,\ldots,X_d]$.
\end{Ass*}

\section{Closed Neighborhood Ideals, Domination, and Decompositions} \label{sec210425a}

In this section, we start by introducing the closed neighborhood ideal of our graph $G$. Then we describe 
its irredundant irreducible decomposition in terms of the dominating sets in Proposition~\ref{thm210426a}. 

\begin{Def}
The \textbf{closed neighborhood} of a vertex $X_i$ is 
\[ N(X_i) = \{X_j \mid X_iX_j \in E\}\cup\{X_i\} \subseteq V\]
which includes $X_i$ by definition.
The \textbf{closed neighborhood monomial} of $X_i$ is the square-free monomial
\[ S_i = \prod_{X_j \in N(X_i)} X_j \in R\]
which by definition is divisible by $X_i$.
The \textbf{closed neighborhood ideal} of $G$ is the ideal $N_G \subseteq R$ generated by all of the closed neighborhood monomials of $G$:
\[ N_G = \left( \left\{ S_i \mid v_i \in V \right\} \right) R\subseteq R. \]
\end{Def}
\begin{Exa}
\label{DominationIdealExample}
We compute the closed neighborhood ideal of the following graph 
\begin{equation}
\begin{split}
\xymatrix{
	{X_1} \ar@{-}[r] \ar@{-}[d] & {X_2} \ar@{-}[r] \ar@{-}[d] & {X_3} \ar@{-}[d] \\
	{X_4} & {X_5} & {X_6} 	
}
\label{graph}
\end{split}
\end{equation}
by applying the definition then removing redundant generators:
\begin{align*}
N_G 
&= \left(X_1X_2X_4,X_1X_2X_3X_5,X_2X_3X_6,X_1X_4,X_2X_5,X_3X_6\right)R\\
&= \left(X_1X_4,X_2X_5,X_3X_6\right)R. 
\end{align*}
In particular, the generating set defining  closed neighborhood ideals can be redundant. 
Furthermore, the following graph has the same closed neighborhood ideal as the one for~\eqref{graph}
\begin{equation}
\begin{split}
\xymatrix{
	& {X_4} \ar@{-}[r] & {X_1} \ar@{-}[dr] \ar@{-}[dl] & & \\
	{X_5} & {X_2} \ar@{-}[rr] \ar@{-}[l] & & {X_3} \ar@{-}[r] & {X_6} \\	
}
\label{graph2}
\end{split}
\end{equation}
showing that the association $G\mapsto N_G$ is not injective.
\end{Exa}

Next, we move toward our decomposition result. Here is the relevant definition from graph theory for this 
purpose. See~\cite{MR4180624,MR1605684} for background on this notion. 

\begin{Def}
A \textbf{dominating set} for $G$ is a subset $V' \subseteq V$ such that for each  $X_i \in V$, we have $V'\cap N(X_i)\neq\emptyset$, i.e., either $X_i \in V'$ or $X_i$ is adjacent to a vertex $X_j \in V'$. A dominating set is \textbf{minimal} if it does not properly contain another dominating set. 
\label{dominatingset}
\end{Def}

This definition is deceptively similar to the definition of a vertex cover; see Definition~\ref{vertexCover}. However, they are different, as the next example shows. The key difference
is that dominating sets dominate vertices while vertex covers dominate edges. 

\begin{Exa}
For our graph \eqref{graph}, the set 
$V' = \{X_3,X_4,X_5,X_6\}$
is a dominating set, but it is not minimal since 
$V'' = \{X_4,X_5,X_6\}$ also dominates $G$. The dominating set $V''$ is minimal.
Note that $V''$ is not a vertex cover of $G$; see Definition~\ref{vertexCover}. 
\label{DominatingSetExample}
\end{Exa}

\begin{Rem}
The full vertex set $V$ is a dominating set for $G$. If $G$ has no isolated vertices, then every subset $V\smallsetminus\{X_i\}$ is a dominating set for $G$.
Since $V$ is finite, every dominating set for $G$ contains a minimal dominating set. 
\label{containsmin}
\end{Rem}

%

%

Here is the aforementioned decomposition result, which essentially follows from~\cite[Lemma~2.2]{MR4132629}.

\begin{Pro}\label{thm210426a}
The  ideal $N_G \subseteq R$ has the following irreducible decompositions
\[
N_G = \bigcap \limits_{V'} (V')R = \bigcap \limits_{V' \text{ min.}} (V')R
\]
where the first intersection is taken over all dominating sets for $G$, and the second intersection is taken over all minimal dominating sets for $G$. Furthermore, the second intersection is irredundant.
\end{Pro}

\begin{proof}
The second intersection is irredundant by the fact that the intersection is taken over only minimal dominating sets, so any redundancies have already been removed.
	
The containment $\bigcap_{V'} (V')R \subseteq \bigcap_{V' \text{ min.}} (V')R$ follows since the first intersection is over a possibly larger index.
	
The containment $\bigcap_{V'} (V')R \supseteq \bigcap_{V' \text{ min.}} (V')R$ follows since every dominating set contains a minimal dominating set by Remark \ref{containsmin}.
	
By~\cite[Lemma~2.2]{MR4132629} we know that $N_G \subseteq (V')R$ if and only if $V'$ dominates $G$, so we have  $N_G \subseteq \bigcap_{V'} (V')R$ where the intersection is taken over the dominating sets of $G$. 
	
To show that $N_G \supseteq \bigcap_{V'} (V')R$,  recall that $N_G$ is square-free. Therefore there exist subsets $V_1, \ldots, V_n\subseteq V$ such that
\[ 
N_G = \bigcap \limits_{i = 1}^n (V_i)R.
\]
In particular, each $V_i$ satisfies $N_G \subseteq (V_i)R$ so each $V_i$ is a dominating set of $G$ by~\cite[Lemma~2.2]{MR4132629}. Since each is a dominating set, and the intersection $\bigcap_{V'} (V')R$ is taken over all dominating sets, we must have 
\[
N_G = \bigcap \limits_{i = 1}^n (V_i)R \supseteq \bigcap \limits_{V'} (V')R
\]
which completes the proof.
\end{proof}

\begin{Exa}\label{eq210426a}
Consider the closed neighborhood ideal of $N_G$  where $G$ is the graph~\eqref{graph} from Example \ref{DominationIdealExample}. It is straightforward to verify the following
irredundant irreducible decomposition algebraically by hand or using Macaulay2~\cite{m2}.
\begin{align*}
N_G =& \left(X_1X_4,X_2X_5,X_3X_6\right)R\\
=&(X_1,X_2,X_3)R\cap (X_1,X_2,X_6)R\cap(X_1,X_5,X_3)R\cap(X_4,X_2,X_3)R\\
&\cap(X_1,X_5,X_6)R\cap(X_4,X_2,X_6)R\cap(X_4,X_5,X_3)R\cap(X_4,X_5,X_6)R.
 \end{align*}
From this, one can use Proposition~\ref{thm210426a} to read the eight minimal dominating sets for $G$: 
\begin{align*}
\{X_1,X_2,X_3\}&& \{X_1,X_2,X_6\}&&\{X_1,X_5,X_3\}&&\{X_4,X_2,X_3\}\\
\{X_1,X_5,X_6\}&&\{X_4,X_2,X_6\}&&\{X_4,X_5,X_3\}&&\{X_4,X_5,X_6\}.
\end{align*}
On the other hand, one can identify the above the list of minimal dominating sets for $G$ combinatorially by inspecting $G$,
then invoke Proposition~\ref{thm210426a} to obtain the decomposition. 

Furthermore, the next result shows that graph~\eqref{graph2} has the same list of minimal dominating sets as $G$.
\end{Exa}

The next result is potentially useful since it is expensive to compute all the minimal dominating sets of a graph, and it feels less expensive to 
check equality of closed neighborhood ideals, which can be accomplished by comparing the  minimal generating sets.

\begin{Cor}\label{cor210426b}
Let $H$ be another graph with vertex set $V$. Then $N_G=N_H$ if and only if $G$ and $H$ have the same lists of (minimal) dominating sets.
\end{Cor}

\begin{proof}
Since the set of dominating sets for a graph is closed under taking supersets, we see that
$G$ and $H$ have the same lists of  dominating sets if and only if $G$ and $H$ have the same lists of minimal dominating sets.

For the backwards implication in the corollary, assume that $G$ and $H$ have the same lists of  dominating sets.
Then Proposition~\ref{thm210426a} implies that 
$$N_G = \bigcap \limits_{\text{$V'$ dom.\ set $G$}} (V')R =\bigcap \limits_{\text{$V'$ dom.\ set $H$}} (V')R =N_H.$$

For the converse, assume that $N_G=N_H$. Then the uniqueness of the ideals in an irredundant irreducible decomposition of this ideal
works with Proposition~\ref{thm210426a} to imply that  
$G$ and $H$ have the same lists of minimal dominating sets, as desired.
\end{proof}

\section{Cohen-Macaulayness of Closed Neighborhood Ideals of Trees}\label{sec210425c}

In this section, we characterize the trees whose closed neighborhood ideals are Cohen-Macaulay; see Theorem~\ref{thm210426b}
and Corollary~\ref{cor210426a}. As in~\cite{villarreal:cmg}, our proof utilizes the following notion from graph theory.

\begin{Def}[Matching]
A \textbf{matching} is a set of pairwise non-adjacent edges of a graph.
A matching is  \textbf{maximal} if it is not properly contained in another matching. 
\end{Def}

\begin{Exa}
In our graph \eqref{graph}, some maximal matchings are
\begin{align*}
\{X_1X_2,X_3X_6\} 
&&\text{and}&&
\{X_1X_4,X_2X_5,X_3X_6\} .
\end{align*}
\end{Exa}

Our Cohen-Macaulayness result is stated in terms of the following notion. Note that more general $H$-coronas can be defined, but we only need $K_1$-coronas.
In a small amount of the literature, these are also referred to as ``suspensions.''

\begin{Def}\label{defn210426a}
A \textbf{$K_1$-corona} of the graph $G$ is the graph $G'$ created by adding a ``whisker'' to each vertex of $G$, i.e., it is a graph with vertex set
$\{X_1,\ldots,X_d,Y_1,\ldots,Y_d\}$ and edge set $E\cup\{X_1Y_1,\ldots,X_dY_d\}$.
\end{Def}

\begin{Exa}
The two graphs in Example~\ref{DominationIdealExample} are $K_1$-coronas of the path $P_2$ and the cycle $C_3$, respectively.
\end{Exa}

The next result states that $K_1$-coronas are particularly nice with respect to minimal dominating sets.
It uses the following notion which we introduce for convenience.

\begin{Def}
We write that $G$ is \textbf{domination-unmixed} if all the minimal dominating sets of $G$ have the same size.
\end{Def}

\begin{Exa}
The graphs~\eqref{graph} and~\eqref{graph2} are domination-unmixed by Example~\ref{eq210426a}.
\end{Exa}

\begin{Pro}
If  $H$ is a $K_1$-corona of $G$, then every minimal dominating set for $H$ has exactly $d$ elements and $N_H$ is a complete intersection;
in particular, $N_H$ is unmixed and $H$ is domination-unmixed.
\label{k1tounmixed}
\end{Pro}

\begin{proof}
Use the notation of Definition~\ref{defn210426a} for the $K_1$-corona $H$. 
Then each monomial $X_iY_i$ is a generator of $N_H$. The other generators are of the form $X_iY_iT_i$ for some monomial $T_i$, so these generators are redundant.
Therefore the closed neighborhood ideal of $H$ is 
\[ N_H = (X_1Y_1,\ldots,X_dY_{d})R \]
and this decomposes into the intersection
\[ N_H = \bigcap (Z_{1},\ldots,Z_{d})R \]
where the intersection is over all $2^d$ combinations when choosing $Z_i$ to be $X_i$ or $Y_i$. In each case, the corresponding minimal dominating set 
from Proposition~\ref{thm210426a} has size $d$. 
\end{proof}
We next provide a few lemmas for subsequent use. The first one uses the following notion mentioned in the introduction.

\begin{Def}
A \textbf{vertex cover} for $G$ is a subset $V' \subseteq V$ such that for each edge $X_iX_j$ in $G$, we have $V'\cap \{X_i,X_j\}\neq\emptyset$, i.e., either $X_i \in V'$ or $X_j \in V'$. A vertex cover is \textbf{minimal} if it does not properly contain another vertex cover. 
\label{vertexCover}
\end{Def}

The definitions of dominating set and vertex cover look similar, but they are crucially different. 
One way to summarize the difference is that
a dominating set dominates every \emph{vertex} while a vertex cover covers every \emph{edge}. 
On the other hand, one actual similarity is in the next result. 

\begin{Lem}
If $G$ has no isolated vertices, then every vertex cover $V'$ of $G$ is a dominating set for $G$.
\label{vertodom}
\end{Lem}
\begin{proof}
By assumption, every edge of $G$ is adjacent to at least one vertex in $V'$. Since $G$ has no isolated vertices, it must be that every vertex is either in $V'$ or adjacent to an edge which is adjacent to a vertex in $V'$. Therefore, $V'$ is a dominating set.
\end{proof}

The converse of Lemma~\ref{vertodom} fails in general, as we see next.

\begin{Exa}
In the graph~\eqref{graph}, the set $\{X_4,X_5,X_6\}$ is a dominating set but not a vertex cover.
\end{Exa}

Our final lemma is  similar  to K\H onig's Theorem. 

\begin{Lem}
For a domination-unmixed bipartite graph $G$ with no isolated vertices, each minimal dominating set has the same size as the largest maximal matching.
\label{konig}
\end{Lem}
\begin{proof}
Since $G$ is bipartite, let $V=V_1\cup V_2$ be a bipartition of $V$. Note that $V_1$ and $V_2$ are dominating sets of $G$ since every vertex is either in $V_i$ or adjacent to an element in $V_i$.
Furthermore, $V_i$ is a minimal dominating set for $G$ since the removal any vertex from $V_i$ causes that vertex not to  be dominated. 
By assumption, this implies that $|V_1| = |V_2| = k$. 

We claim that every vertex cover $V'$ of $G$ has size $|V'| \ge k$. To show this, we note that $V'$ is a dominating set by Lemma \ref{vertodom}. Therefore, it contains a minimal dominating set $V'' \subseteq V'$. Then $|V'| \ge |V''| = k$, as claimed.

Therefore we have that the smallest minimal vertex cover has size $k$, and therefore by K\H onig's Theorem, we know that the largest maximal matching has size $k$.
\end{proof}

We conclude with the main results of this section. 

\begin{The}\label{thm210426b}
A tree $T$ is a $K_1$-corona  if and only if it is domination-unmixed.
\end{The}
\begin{proof}
The forward implication follows from Proposition~\ref{k1tounmixed}.

For the converse, assume that $T$ is domination-unmixed. Since $T$ is a tree, it is bipartite. Let $V_1$, $V_2$, and $k$ be as in the proof of Lemma~\ref{konig}; this result implies that the largest maximal matching must have size $k$. Let $\{x_1y_1,\ldots,x_ky_k\}$ be such a matching.

We next need to show that for  $i = 1,\ldots,k$, either $\deg{x_i} = 1$ or $\deg{y_i} = 1$. Suppose this is not true, and reorder the vertices if necessary to assume that $\deg(x_1),\deg(y_1)\geq 2$. Then we have
$N(x_1)=\{x_1,a_1,\ldots,a_p\}$ and $N(y_1)=\{y_1,b_1,\ldots,b_q\}$ with $p,q \ge 1$. 

Assume without loss of generality that $x_1 \in V_1$ and $y_1 \in V_2$. 
Each set $V_i$ is a minimal dominating set of size $k = p + q + m + 1$. Since  $T$ is domination-unmixed,  this is the size of the smallest minimal dominating set.

We partially sketch $T$ as  
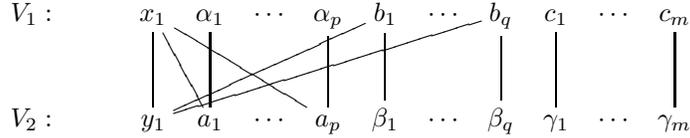
\[
\xymatrix@C=0.5em{V_1: &&&
	{x_1} \ar@{-}[d] \ar@{-}[dr] \ar@{-}[rrrd] & {\alpha_1} \ar@{-}[d] & {\cdots} & {\alpha_p} \ar@{-}[d] & {b_1} \ar@{-}[d] & {\cdots} & {b_q} \ar@{-}[d] & {c_1} \ar@{-}[d] & {\cdots} & {c_m} \ar@{-}[d] \\
	V_2:&&&
	{y_1} \ar@{-}[rrrru] \ar@{-}[rrrrrru] & {a_1} & {\cdots} & {a_p} & {\beta_1} & {\cdots} & {\beta_q} & {\gamma_1} & {\cdots} & {\gamma_m}
}
\]
where the vertical edges comprise the given maximal matching.
Also note that if $m > 0$, then for each $i$, one of the two vertices $c_i$ or $\gamma_i$ is adjacent to some $a_j$, $b_j$, $\alpha_j$, $\beta_j$, $c_j$, or $\gamma_j$ because $T$ is connected and neither $c_i$ nor $\gamma_i$ can be adjacent to $x_1$ or $y_1$ by construction. We can order the vertices $\{c_1,\ldots,c_m\}$ and $\{\gamma_1,\ldots,\gamma_m\}$ so that if $c_i$ or $\gamma_i$ is adjacent to $c_j$ or $\gamma_j$, then $j < i$. Note that no $c_i$ can be adjacent to another $c_j$, and likewise for any $\gamma_i$.

It is straightforward to show that the set $\{a_1,\ldots,a_p,b_1,\ldots,b_q,c_1,\ldots,c_m\}$ dominates $T$.
This is a set of size $p + q + m$, and a minimal dominating set contained in it cannot be larger. Since this is smaller than $p + q + m + 1$,  we have a contradiction since we assumed $T$ is domination-unmixed. Therefore $\deg{x_i} = 1$ or $\deg{y_i} = 1$, and we have that $T$ is a $K_1$-corona.
\end{proof}

\begin{Cor}
\label{cor210426a}
For a tree $T$, the following conditions are equivalent:
\begin{enumerate}[\rm (i)]
\item\label{cor210426a4} $T$ is domination-unmixed,
\item\label{cor210426a1} $N_T$ is unmixed,
\item\label{cor210426a2} $N_T$ is Cohen-Macaulay, and
\item\label{cor210426a3} $N_T$ is a complete intersection.
\end{enumerate}
In particular, the Cohen-Macaulay property for $N_T$ is characteristic-independent.
\end{Cor}

\begin{proof}
The equivalence $\eqref{cor210426a4}\iff\eqref{cor210426a1}$ is from Proposition~\ref{thm210426a}, and the implications  
$\eqref{cor210426a3}\implies\eqref{cor210426a2}\implies\eqref{cor210426a1}$ are standard.
And the implication $\eqref{cor210426a4}\implies\eqref{cor210426a3}$ follows from
Proposition~\ref{k1tounmixed} and Theorem~\ref{thm210426b}.
\end{proof}

\section*{Acknowledgments}
We are grateful to Mehrdad Nasernejad for pointing us to~\cite{MR4132629}. 


\begin{thebibliography}{10}

\bibitem{morey:dcmppi}
D.~Campos, R.~Gunderson, S.~Morey, C.~Paulsen, and T.~Polstra, \emph{Depths and
  {C}ohen-{M}acaulay properties of path ideals}, J. Pure Appl. Algebra
  \textbf{218} (2014), no.~8, 1537--1543. \MR{3175038}

\bibitem{conca:msgililt}
A.~Conca and E.~De~Negri, \emph{{$M$}-sequences, graph ideals, and ladder
  ideals of linear type}, J. Algebra \textbf{211} (1999), no.~2, 599--624.
  \MR{1666661 (2000d:13020)}

\bibitem{diestel:gt}
R.~Diestel, \emph{Graph theory}, fourth ed., Graduate Texts in Mathematics,
  vol. 173, Springer, Heidelberg, 2010. \MR{2744811}

\bibitem{francisco:wscmg}
C.~A. Francisco and H.~T. H{\`a}, \emph{Whiskers and sequentially
  {C}ohen-{M}acaulay graphs}, J. Combin. Theory Ser. A \textbf{115} (2008),
  no.~2, 304--316. \MR{2382518 (2008j:13050)}

\bibitem{francisco:scmei}
C.~A. Francisco and A.~Van~Tuyl, \emph{Sequentially {C}ohen-{M}acaulay edge
  ideals}, Proc. Amer. Math. Soc. \textbf{135} (2007), no.~8, 2327--2337
  (electronic). \MR{2302553 (2008a:13030)}

\bibitem{m2}
D.~R. Grayson and M.~E. Stillman, \emph{Macaulay 2, a software system for
  research in algebraic geometry}, Available at
  \href{http://www.math.uiuc.edu/Macaulay2/}%
  {http://www.math.uiuc.edu/Macaulay2/}.

\bibitem{MR4180624}
T.~W. Haynes, S.~T. Hedetniemi, and M.~A. Henning (eds.), \emph{Topics in
  domination in graphs}, Developments in Mathematics, vol.~64, Springer, Cham,
  [2020] \copyright 2020. \MR{4180624}

\bibitem{MR1605684}
T.~W. Haynes, S.~T. Hedetniemi, and P.~J. Slater, \emph{Fundamentals of
  domination in graphs}, Monographs and Textbooks in Pure and Applied
  Mathematics, vol. 208, Marcel Dekker, Inc., New York, 1998. \MR{1605684}

\bibitem{MR1088576}
S.~T. Hedetniemi and R.~C. Laskar, \emph{Bibliography on domination in graphs
  and some basic definitions of domination parameters}, Discrete Math.
  \textbf{86} (1990), no.~1-3, 257--277. \MR{1088576}

\bibitem{MR2231097}
J.~Herzog, T.~Hibi, and X.~Zheng, \emph{Cohen-{M}acaulay chordal graphs}, J.
  Combin. Theory Ser. A \textbf{113} (2006), no.~5, 911--916. \MR{2231097}

\bibitem{kubik:piwg}
B.~Kubik and S.~Sather-Wagstaff, \emph{Path ideals of weighted graphs}, J. Pure
  Appl. Algebra \textbf{219} (2015), no.~9, 3889--3912. \MR{3335988}

\bibitem{paulsen:eiwg}
C.\ Paulsen and S.\ Sather-Wagstaff, \emph{Edge ideals of weighted graphs}, J.
  Algebra Appl. \textbf{12} (2013), no.~5, 24 pp. \MR{3055580}

\bibitem{MR4132629}
L.~Sharifan and S.~Moradi, \emph{Closed neighborhood ideal of a graph}, Rocky
  Mountain J. Math. \textbf{50} (2020), no.~3, 1097--1107. \MR{4132629}

\bibitem{villarreal:cmg}
R.~H. Villarreal, \emph{Cohen-{M}acaulay graphs}, Manuscripta Math. \textbf{66}
  (1990), no.~3, 277--293. \MR{1031197 (91b:13031)}

\bibitem{villarreal:ma}
\bysame, \emph{Monomial algebras}, Monographs and Textbooks in Pure and Applied
  Mathematics, vol. 238, Marcel Dekker Inc., New York, 2001. \MR{1800904
  (2002c:13001)}

\end{thebibliography}
\providecommand{\bysame}{\leavevmode\hbox to3em{\hrulefill}\thinspace}
\providecommand{\MR}{\relax\ifhmode\unskip\space\fi MR }
\providecommand{\MRhref}[2]{%
  \href{http://www.ams.org/mathscinet-getitem?mr=#1}{#2}
}
\providecommand{\href}[2]{#2}

\bibliographystyle{amsplain}

\end{document}